\documentclass[11pt]{amsart}

\pdfoutput=1

\usepackage{xcolor}

\usepackage{amsmath} \usepackage[foot]{amsaddr}

\usepackage{appendix}
\usepackage{cancel}
\usepackage{esint,amssymb} 
\usepackage{graphicx}
\usepackage{MnSymbol}
\usepackage{mathtools} 
\usepackage[colorlinks=true, pdfstartview=FitV, linkcolor=blue, citecolor=blue, urlcolor=blue,pagebackref=false]{hyperref}
\usepackage{orcidlink}
\usepackage{microtype}

\usepackage{bm}
\usepackage{dsfont}
\usepackage{mathrsfs}

\usepackage{enumitem}

\definecolor{darkgreen}{rgb}{0,0.5,0}
\definecolor{darkblue}{rgb}{0,0,0.7}
\definecolor{darkred}{rgb}{0.9,0.1,0.1}

\newtheorem{proposition}{Proposition}
\newtheorem{theorem}[proposition]{Theorem}
\newtheorem{lemma}[proposition]{Lemma}
\newtheorem{corollary}[proposition]{Corollary}

\newtheorem{remark}[proposition]{Remark}
\theoremstyle{remark}

\theoremstyle{definition}
\newtheorem{definition}[proposition]{Definition}

\numberwithin{equation}{section}
\numberwithin{proposition}{section}
\numberwithin{figure}{section}
\numberwithin{table}{section}

\newcommand{\N}{\mathbb{N}}

\newcommand{\R}{\mathbb{R}}

\newcommand{\E}{\mathbb{E}}

\newcommand{\ellipt}{\mathsf{Ellipt}}

\newcommand{\law}{\mathsf{law}}

\renewcommand{\le}{\leqslant}
\renewcommand{\ge}{\geqslant}
\renewcommand{\leq}{\leqslant}
\renewcommand{\geq}{\geqslant}

\renewcommand{\subset}{\subseteq}
\renewcommand{\bar}{\overline}

\newcommand{\Ll}{\left}
\newcommand{\Rr}{\right}
\renewcommand{\d}{\mathrm{d}}

\newcommand{\1}{\mathbf{1}}

\newcommand{\mcl}{\mathcal}

\newcommand{\msc}{\mathscr}

\newcommand{\eps}{\varepsilon}
\newcommand{\si}{\sigma}
\newcommand{\Id}{\mathrm{Id}}

\newcommand{\upa}{\uparrow}

\newenvironment{e}{\begin{equation}}{\end{equation}\ignorespacesafterend}
\newenvironment{e*}{\begin{equation*}}{\end{equation*}\ignorespacesafterend}

\newenvironment{a*}{\begin{align*}}{\end{align*}\ignorespacesafterend}

\begin{document}

\author{Hong-Bin Chen\,\orcidlink{0000-0001-6412-0800}}
\address[Hong-Bin Chen]{Institut des Hautes Études Scientifiques, Bures-sur-Yvette, France}
\email{\href{mailto:hongbin.chen@nyu.edu}{hongbin.chen@nyu.edu}}

\author{Victor Issa\,\orcidlink{0009-0009-1304-046X}}
\address[Victor Issa]{Department of Mathematics, ENS de Lyon, Lyon, France}
\email{\href{mailto:victor.issa@ens-lyon.fr}{victor.issa@ens-lyon.fr}}

\subjclass[2010]{}

\title[Uniqueness of Parisi measures]{Uniqueness of Parisi measures for enriched convex vector spin glass}

\begin{abstract}
    In the PDE approach to mean-field spin glasses, it has been observed that the free energy of convex spin glass models could be enriched by adding an extra parameter in its definition, and that the thermodynamic limit of the enriched free energy satisfies a partial differential equation. This parameter can be thought of as a matrix-valued path, and the usual free energy is recovered by setting this parameter to be the constant path taking only the value $0$. Furthermore, the enriched free energy can be expressed using a variational formula, which is a natural extension of the Parisi formula for the usual free energy. 

    For models with scalar spins the Parisi formula can be expressed as an optimization problem over a convex set, and it was shown in \cite{auffinger2015parisi} that this problem has a unique optimizer thanks to a strict convexity property. For models with vector spins, the Parisi formula cannot easily be written as a convex optimization problem.
    
    In this paper, we generalize the uniqueness of Parisi measures proven in \cite{auffinger2015parisi} to the enriched free energy of models with vector spins when the extra parameter is a strictly increasing path. Our approach relies on a Gateaux differentiability property of the free energy and the envelope theorem.

    \bigskip


    \medskip


\end{abstract}

\maketitle

\newpage 
\thispagestyle{empty}
{
  \hypersetup{linkcolor=black}
  \tableofcontents
}

\newpage
\pagenumbering{arabic}
\section{Introduction}
\subsection{Preamble}
In this paper, we study centered Gaussian processes $H_N$ on $(\R^N)^D$ with the following covariance structure,
\begin{e}\label{e.H_N(sigma)=}
    \E H_N(\sigma) H_N(\tau) = N \xi \left( \frac{\sigma \tau^*}{N} \right),
\end{e}
where $\xi \in \mcl C^\infty(\R^{D \times D},\R)$ is a fixed smooth function and where $\sigma \tau^*$ denotes the matrix of scalar products 
\begin{e} \label{e.overlapmatrix}
   \sigma \tau^*  = (\sigma_d \cdot \tau_{d'})_{1 \leq d,d' \leq D} \in \R^{D \times D},
\end{e}
with $\sigma_d = (\sigma_{di})_{1 \leq i \leq N}$. 

We often identify $(\R^N)^D$ with $(\R^D)^N$ and also $\R^{D \times N}$, the set of $D \times N$ matrices, which makes the notation for \eqref{e.overlapmatrix} natural. 
Let $S^D$, $S^D_+$, and $S^D_{++}$ denote the set of $D \times D$ symmetric matrices, positive semi-definite $D \times D$ symmetric matrices, and positive definite $D \times D$ symmetric matrices respectively. For $a,b \in S^D_+$, we write $a \geq b$ whenever $a-b \in S^D_+$.

Throughout, we will assume that $\xi$ admits a convergent power series and is convex on $S^D_+$. When $D = 1$, this setup corresponds to the usual setup for mean-field spin glass models with scalar spins as studied in \cite{pan}. For $D > 1$, this setup corresponds to mean-field spin glass models with $D$-dimensional vector spins. Multi-species models \cite{pan.multi} and the Potts model \cite{pan.potts} are examples of mean-field spin glass models with vector spins.

We give ourselves $P_1$, a compactly supported probability measure on $\R^D$. Without loss of generality, we may assume that the support is included in the unit ball of $\R^D$. We let $P_N = P_1^{\otimes N}$ denote the law of an element $\sigma \in (\R^D)^N$ with independent rows $\sigma_i =(\sigma_{di})_{1 \leq d \leq D} \in \R^D$ with law $P_1$. We are interested in the large-$N$ limit of 
\begin{e} \label{e.free energy}
    \bar F_N(t,0) = -\frac{1}{N} \E \log \int \exp \left( \sqrt{2t} H_N(\sigma) - t N \xi \left( \frac{\sigma \sigma^*}{N}\right) \right) \d P_N(\sigma),
\end{e}
where $t \ge 0$. The term $\xi\Ll( \frac{\si \si^*}{N} \Rr)$ in \eqref{e.free energy} is introduced as a convenience to simplify the expression of the limit; it is constant in classical cases of interest, such as when the coordinates of $\sigma$ take values in $\{-1,1\}$ and $\xi$ depends only on the diagonal entries of its argument. In general, the second argument of $\bar F_N$ can be any $q$ in the space of nondecreasing functions from $[0,1)$ to $S^D_+$, subject to a mild integrability and continuity requirement; the expression in \eqref{e.free energy} is with this argument chosen to be the constant path taking only the value $0 \in S^D_+$. To explain what this space is, let us say that a path $q : [0,1) \to S^D_+$ is nondecreasing if, for every $u \le v \in [0,1)$, we have $q(v) - q(u) \in S^D_+$. We then let 
\begin{e}\label{e.Q=}
    \mcl Q = \{ q : [0,1) \to S^D_+\ \big| \ \text{$q$ is nondecreasing and càdlàg}\}
\end{e}
and we set $\mcl Q_p = \mcl Q \cap L^p([0,1);S^D)$. We postpone the precise definition of $\bar F_N(t,q)$ for arbitrary $q \in \mcl Q_2$ to \eqref{e.enriched free energy}. In short, this quantity is obtained by adding an energy term in the exponential on the right side of \eqref{e.free energy} to encode the interaction of $\sigma$ with an external magnetic field, and this external magnetic field has an ultrametric structure whose characteristics are encoded by the path $q$.
Therefore, we call $\bar F_N(t,q)$ the \textbf{enriched free energy}.

One can check \cite[Proposition~3.1]{mourrat2019parisi} that $\bar F_N(0,\cdot)$ does not depend on $N$. This follows from the fact $P_N = P_1^{\otimes N}$ and that at $t = 0$ the $N$-body Hamiltonian has the same law as $N$ copies of the $1$-body Hamiltonian. For every $q \in \mcl Q_2$, we write
\begin{equation}  \label{e.psi=}
    \psi(q) = \bar F_1(0,q) = \bar F_N(0,q). 
\end{equation}
When instead $P_N$ is the uniform measure on the sphere of radius $\sqrt{N}$ centered at $0$ in $(\mathbb{R}^D)^N$, $\bar F_N(0,\cdot)$ depends on $N$ but converges to a differentiable function of $q$ as $N \to +\infty$ \cite[Proposition~3.1]{mourrat2019parisi}. In what follows, we focus on models with $P_N = P_1^{\otimes N}$.

When $\xi$ is convex on $S^D_+$, the limiting value of $\bar F_N(t,0)$ is known, this is the celebrated Parisi formula. The Parisi formula was first conjectured in \cite{parisi1979infinite} using a sophisticated non-rigorous argument known as the replica method. The convergence of the free energy as $N \to +\infty$ was rigorously established in \cite{guerra2002} in the case of the so-called Sherrington-Kirkpatrick model which corresponds to $D = 1$, $\xi(x) =\beta^2 x^2$ for $\beta>0$, and $P_1 = \mathrm{Unif}(\{-1,1\})$. The Parisi formula for the Sherrington-Kirkpatrick model was then proven in \cite{gue03,Tpaper}. This was extended to the case $D = 1$, $P_1 = \mathrm{Unif}(\{-1,1\})$ and $\xi(x) = \sum_{p \geq 1} \beta^2_p x^p$ with $\beta_p \geq 0$ in \cite{pan}. Some models with $D > 1$ such as multi-species models, the Potts model, and a general class of models with vector spins were treated in \cite{pan.multi,pan.potts,pan.vec}, under the assumption that $\xi$ is convex on $\R^{D \times D}$. Finally, the case $D > 1$ was treated in general in \cite{chenmourrat2023cavity} assuming only that $\xi$ is convex on the set of positive semi-definite matrices. In addition, following \cite{chenmourrat2023cavity, mourrat2020extending}, the Parisi formula is extended to $\lim_{N \to +\infty} \bar F_N(t,q)$ for $q  \in \mcl Q_2$. The following version of the Parisi formula is \cite[Proposition~8.1]{chenmourrat2023cavity}.
Throughout, under the assumption that $\xi$ is convex on $S^D_+$, for every $t\geq 0$ and $q\in \mcl Q_2$, we set
\begin{align}\label{e.f=}
    f(t,q) = \lim_{N\to\infty} \bar F_N(t,q).
\end{align}
We view $f$ as a real-valued function on $[0,+\infty)\times \mcl Q_2$ and we often consider $f$ restricted to a smaller domain. We set
\begin{gather}
    L^\infty_{\leq 1} = \Ll\{p:[0,1)\to S^D \ \big| \ \text{$|p(u)| \leq 1$ almost everywhere}\Rr\},\label{e.L^infty_<1}
    \\
    \theta(x) = x \cdot \nabla \xi(x) - \xi(x). \label{e.theta=}
\end{gather}
\begin{theorem} [Generalized Parisi formula \cite{chenmourrat2023cavity}] \label{t.parisi}
    If $\xi$ is convex on $S^D_+$, then, at every $t \geq 0$ and $q \in \mcl Q_2$, the limit of the free energy $\bar F_N(t,q)$ is given by
\begin{e} \label{e.parisi}
        f(t,q) = \sup_{p \in \mcl Q \cap L^\infty_{\leq 1}} \left\{ \psi(q + t \nabla \xi \circ p )  - t \int_0^1 \theta(p(u)) \d u \right\}.    
    \end{e}
\end{theorem}
Setting $q = 0$ in \eqref{e.parisi}, we recover the usual Parisi formula as presented in \cite{pan} for example. Letting $\xi^* : \R^{D \times D} \to \R \cup\{+\infty\}$ denote the convex dual of $\xi$ with respect to the cone $S^D_+$, that is $\xi^*$ is the function defined by
\begin{equation} \label{e.def.xi.star}
    \xi^*(a) = \sup_{b \in S^D_+} \Ll\{a \cdot b - \xi(b)\Rr\},\forall a \in \R^{D\times D},
\end{equation}
the function $\theta$ can be rewritten as $\theta(x) = \xi^*(\nabla \xi(x))$ (see Lemma~\ref{l.theta} below). 
\begin{definition}[Parisi measures]\label{d.parisi_measure}
We say that a path $p \in \mcl Q\cap L^\infty_{\leq 1}$ is a \emph{Parisi measure} at $(t,q)$ when it is an optimizer in the right-hand side of \eqref{e.parisi}. 
\end{definition}
Usually in the literature, the term Parisi measure refers to the law of $p(U)$ where $U$ is a uniform random variable in $[0,1)$ and $p$ is an optimizer in the right-hand side of \eqref{e.parisi}. Here since we work in the space of paths rather than in the space of probability measures, we choose to use the term Parisi measure to refer to $p$ directly. 

Studying the set of Parisi measures is an important question in itself since they are the optimizers in the fundamental Parisi formula. Their study is further motivated by the fact that they describe the limit in law of the overlap matrix under the expected Gibbs measure, that is the limit in law of $\sigma\tau^*/N$ as $N \to +\infty$ where $\sigma$ and $\tau$ are two independent random variables drawn from the expected Gibbs measure. 

A first rigorous proof of the uniqueness of Parisi measures at $(t,0)$ for scalar spins (i.e. $D = 1$) was put forward in \cite{talagrand2006parisimeasures} but relied on genericity assumptions on $\xi$. In \cite{talagrand2006parisimeasures}, the function $\xi$ is written $\xi(x) = \sum_{p} \beta_p^2 x^{2p}$. Assuming that for every $p$, $\beta_p \neq 0$, the author shows the uniqueness of Parisi measures by differentiating the limit free energy and the Parisi formula with respect to each $\beta_p$. In \cite{auffinger2015parisi}, the uniqueness of Parisi measures at $(t,0)$ is proven in the case of scalar spins (i.e. $D =1$) without any further assumptions on $\xi$. In Appendix~\ref{s.auffingerchen} we give some key ideas of the proof of the uniqueness of Parisi measures at $(t,0)$ given in \cite{auffinger2015parisi} in the context of scalar spins. We also show that thanks to a convexity property from optimal transport, with minimal adaptation, this proof also yields the uniqueness of Parisi measures at $(t,q)$ for every $q \in \mcl Q$ when $D = 1$. In addition, we explain why those arguments do not carry over easily to models with vector spins.

\subsection{Main result}

Let $\text{Id}$ be the $D \times D$ identity matrix. Recall that for $a, b\in S^D$ satisfying $a-b \in S^D_+$, we write $a\geq b$. 
We can view $\nabla\xi$ as an $\R^{D\times D}$-valued function. Moreover, we assume that $\nabla\xi(S^D) \subset S^D$.
Because \eqref{e.H_N(sigma)=} holds, $\xi$ enjoys some extra properties. For example as a consequence of \cite[Propositions~6.4~\&~6.6]{mourrat2020free}, we can assume without loss of generality that $\xi$ satisfies the following monotonicity:
\begin{gather}
    a,\, b \in S^D_+,\quad a\geq b \qquad \Longrightarrow \qquad \xi(a)\geq\xi(b). \label{e.xi_monotonicity}
    \\
    a,\, b \in S^D_+,\quad a\geq b \qquad \Longrightarrow \qquad \nabla\xi(a)\geq\nabla\xi(b). \notag  \end{gather}
Throughout, we will also assume that $\xi$ is superlinear on $S^D_+$, which means
\begin{align}\label{e.superlinear}
    \forall M>0,\quad \exists R>0:\qquad \inf_{x \in S^D_+,\ |x|\geq R}\xi(x)/|x|\geq M.
\end{align}
Also, recall that we assume that $\xi$ is convex on $S^D_+$ and admits an absolutely convergent power series expansion. These conditions along with the existence of $H_N$ are satisfied by a wide range of interactions. We refer to~\cite[Section~1.5]{chenmourrat2023cavity} for the detail.

For every $a \in S^D$, let $\lambda_{\max}(a)$ and $\lambda_{\min}(a)$ be the largest and smallest eigenvalues of $a$, respectively. For $a\in S^D_{++}$, we define 
\begin{e}\label{e.Ellipt(a)=}
    \ellipt(a) = \frac{\lambda_{\max}(a)}{\lambda_{\min}(a)}.
\end{e}
We say that $q \in \mcl Q_\upa$ when $q \in \mcl Q_2$, $q(0) = 0$ and there exists a constant $c > 0$ such that for every $u < v$,
\begin{e}\label{e.Q_uparrow=}
    q(v) -q(u) \geq c(v-u) \text{Id} \quad\text{and}\quad \ellipt(q(v)-q(u)) \leq \frac{1}{c}.
\end{e}
Throughout, we set
\begin{align}\label{e.Q_inftyuparrow=}
    \mcl Q_{\infty,\uparrow} = \mcl Q_\infty \cap \mcl Q_\uparrow.
\end{align}
In this paper, we study the uniqueness of Parisi measures at $(t,q)$ for $q \in \mcl Q_2$. We will show the following theorem. 
\begin{theorem}[Uniqueness of Parisi measures] \label{t.uniqueness}
    Assume that $\xi$ is strictly convex on $S^D_+$ and is superlinear on $S^D_+$. For every $(t,q) \in (0,+\infty) \times \mcl Q_{\infty,\uparrow}$, there is a unique Parisi measure at $(t,q)$ and it is given by $p = \nabla_q f(t,q)$. 
\end{theorem}
Here, $\nabla_q f(t,q)$ denotes the Gateaux derivative of $f(t,\cdot)$ at $q$, see Definition~\ref{d.gateaux diff} below. Note that Theorem~\ref{t.uniqueness} does not include the case $q = 0$, which is arguably the most interesting. The proof of Theorem~\ref{t.uniqueness} is in the style of \cite{talagrand2006parisimeasures} and relies on the Gateaux differentiability of the limit free energy on $(0,+\infty) \times \mcl Q_{\infty,\uparrow}$ as proven in \cite[Proposition~8.1]{chenmourrat2023cavity}.
\begin{definition}[Gateaux differentiability] \label{d.gateaux diff}
Let $(E,\,|\,\cdot\,|_E)$ be a Banach space and denote by $\langle \,\cdot\,,\, \cdot\, \rangle_E$ the canonical pairing between $E$ and its dual $E^*$. Let $G$ be a subset of $E$ and, for every $q \in G$, we define
\begin{e*}
    \mathrm{Adm}(G,q) = \{\kappa \in E \big| \, \exists r > 0, \, \forall t\in [0,r], q+t\kappa \in G \},
\end{e*}
to be the set of admissible directions at $q$. A function $h : G \to \R$ is Gateaux differentiable at $q \in G$ if the following two conditions hold:
\begin{enumerate}
    \item For every $\kappa \in \mathrm{Adm}(G,q)$, the following limit exists 
\begin{e*}
        h'(q,\kappa) = \lim_{t \downarrow 0} \frac{h(q+t\kappa)-h(q)}{t}.
    \end{e*}

    \item There is a unique $y^* \in E^*$ such that for every $\kappa \in \mathrm{Adm}(G,q)$, 
\begin{e*}
        h'(q,\kappa) = \langle y^*,\kappa \rangle_E.
    \end{e*}
\end{enumerate}
In such a case, we call $y^*$ the Gateaux derivative of $h$ at $q$ and we denote it $\nabla h(q)$.
\end{definition}
We close the introduction with a few remarks related to Theorem~\ref{t.uniqueness}. First, as a consequence of the cavity computations performed in \cite{chenmourrat2023cavity} the unique Parisi measure encodes the limit law of the overlap.
\begin{remark}[Limit law of the overlap] If $\xi$ satisfies the hypotheses of Theorem~\ref{t.uniqueness}, and we further assume that $\xi$ is strictly convex on $\R^{D \times D}$ or that $\xi$ is strongly convex on $S^D_+$, then it follows from \cite[Proposition~8.8]{chenmourrat2023cavity} and \cite[Theorem~1.4]{chenmourrat2023cavity} respectively, that for every $(t,q) \in (0,+\infty) \times \mcl Q_{\infty,\uparrow}$, the law of the overlap matrix $\sigma\tau^*/N$ under the expected Gibbs measure converges as $N \to +\infty$ to the law of $p(U)$ where $p$ is the unique Parisi measure at $(t,q)$ and $U$ is a uniform random variable in $[0,1)$.
\end{remark}
Furthermore, it follows from \cite[Proposition~5.16]{simulRSB} that when $q \in \mcl Q_{\infty,\upa}$ the unique Parisi measure is strictly increasing.
\begin{remark}[$\infty$-RSB at $q \in \mcl Q_{\infty,\uparrow}$]Assume that $\xi$ is as in Theorem~\ref{t.uniqueness} and that the support of $P_1$ sans $\R^D$. At every $(t,q) \in (0,+\infty)\times\mcl Q_{\infty,\uparrow}$, the unique Parisi measure $p$ is strictly increasing on $[0,1)$ (i.e.\ $p(v)\neq p(u)$ for every distinct $v,u\in [0,1)$. 
\end{remark}
Finally, since the limit free energy is Lipschitz on $\R_+ \times \mcl Q_1$ \cite[Proposition~5.1]{chenmourrat2023cavity} and $\mcl Q_{\infty,\upa}$ is dense in $\mcl Q_1$, we have the following perturbative result.
\begin{remark}[Uniqueness up to small perturbation]
    There is a constant $C>0$ such that, for every $t>0$ and $q\in \mcl Q_1$,
    \begin{align*}
        \Ll|f(t,0)-f(t,q)\Rr|\leq C|q|_{L^1}.
    \end{align*}
    Therefore, for every $\eps>0$, we can find $q\in \mcl Q_{\infty,\uparrow}$ such that $f(t,q)$ differs from $f(t,0)$ by $\eps$ uniformly in $t$. Hence, we can roughly state that, in a convex vector spin glass model, up to an arbitrarily small perturbation the Parisi measure is unique.
\end{remark}
To conclude, we point out the following technical considerations about multi-species models.
\begin{remark}[Adaption of results to multi-species models]
    We refer to~\cite{chen2024free} for a general setup of a multi-species spin glass model.
    If the species population ratios are not rational, then a multi-species model is not equivalent to a vector spin glass model as encoded here. 
    So, the results here are not directly applicable. However, they can be adapted straightforwardly using the same arguments and replacing results cited from~\cite{chenmourrat2023cavity} by corresponding results in~\cite{chen2024free}.
\end{remark}

\subsection{Organization of the paper}
In Section~\ref{s.enriched} we give a precise definition of the enriched free energy. In Section~\ref{ss.hopflax} we prove a modified version of the Hopf--Lax formula derived in \cite[Corollary~8.2]{chenmourrat2023cavity}, see \eqref{e.Hopf--Lax} in Theorem~\ref{t.Hopf--Lax}. In Section~\ref{s.uniqueness} we use this modified Hopf-Lax formula to prove Theorem~\ref{t.uniqueness}. In Section~\ref{ss.frechet}, as an application of Theorem~\ref{t.uniqueness}, we upgrade the Gateaux differentiability result \cite[Proposition~8.1]{chenmourrat2023cavity} to a Fréchet differentiability result, see Definition~\ref{d.frechet} and Theorem~\ref{t.frechet}. Finally, in Section~\ref{s.unique crit point}, we show that Theorem~\ref{t.uniqueness} implies uniqueness in the critical point representation of \cite{chenmourrat2023cavity}. In Appendix~\ref{s.auffingerchen} we show the uniqueness of Parisi measures at $(t,q)$ for $q \in \mcl Q$ and $D =1$ using a different argument relying on the strict convexity property of \cite{auffinger2015parisi}. We also explain why this argument does not carry over easily when $D > 1$.

\subsection*{Acknowledgments}

The second author thanks Jean-Christophe Mourrat for introducing the problem of the uniqueness of Parisi measures for the enriched free energy to him and for many helpful comments and suggestions. The authors also thank Wei-Kuo Chen and Dmitry Panchenko for useful feedback on preliminary versions of this work. The first author is funded by the Simons Foundation.

\section{Definition of the enriched free energy} \label{s.enriched}
 
The goal of this section is to give a precise definition of the enriched free energy appearing in Theorem~\ref{t.parisi} and Theorem~\ref{t.uniqueness}. We start by giving a definition of $F_N(t,q)$ for a piecewise constant $q \in \mcl Q$. We recall that, for $A,B \in \R^{D \times D}$, we write $A \leq B$ when $B-A \in S^D_+$; and $A < B$ when $A \leq B$ and $A \neq B$. Let $K \in \N$ and let $q \in \mcl Q$ be a path of the form
\begin{e} \label{e.piecewise}
    q = \sum_{k = 0}^K q_k \1_{[\zeta_k,\zeta_{k+1})}
\end{e}
where
\begin{gather}
    0 = q_{-1} \leq q_0 < q_1 < \dots < q_K, \label{e.qs=}
    \\
    0 = \zeta_0 < \zeta_1 < \dots < \zeta_K < \zeta_{K+1}=1.\label{e.zetas=}
\end{gather}
The construction of $F_N(t,q)$ involves a random probability measure with ultrametric properties called the Poisson--Dirichlet cascade. We briefly introduce this object and refer to \cite[Section~2.3]{pan} for a more detailed explanation. We define
\begin{e*}
    \mcl A = \N^0 \cup \N^1 \cup \dots \cup \N^K
\end{e*}
with the understanding that $\N^0 = \{ \emptyset \}$. We think of $\mcl A$ as a tree rooted at $\emptyset$ such that each vertex of depth $k < K$ has countably many children. For each $k < K$ and $\alpha = (n_1,\dots,n_k) \in \N^k$ the children of $\alpha$ are the vertices of the form 
\begin{e*}
    \alpha n = (n_1,\dots,n_k,n) \in \N^{k+1}.
\end{e*}
The depth of $\alpha = (n_1,\dots,n_k)$ is denoted by $|\alpha| = k$ and for every $l \leq k$, we write 
\begin{e*}
    \alpha_{|l} = (n_1,\dots,n_l)
\end{e*}
to denote the ancestor of $\alpha$ at depth $l$. Given two leaves $\alpha,\beta \in \N^K$, we denote by $\alpha \wedge \beta$ the generation of the most recent common ancestor of $\alpha$ and $\beta$, that is
\begin{e*}
    \alpha \wedge \beta = \sup \{ k \leq K:\: \alpha_{|k} = \beta_{|k} \}.
\end{e*}
We attach an independent Poisson process to each non-leaf vertex $\alpha \in \mcl A$ with intensity measure. 
\begin{e*}
    x^{-1-\zeta_{|\alpha|+1}} \d x.
\end{e*}
We order increasingly the points of those Poisson processes and denote them by $u_{\alpha 1} \geq u_{\alpha 2} \geq \dots$. For every $\alpha \in \N^K$, we set $w_\alpha = \prod_{k  = 1}^K u_{\alpha_{|k}}$ and define
\begin{e*}
   v_\alpha = \frac{w_\alpha}{\sum_{\beta \in \N^K} w_\beta}. 
\end{e*}
\begin{definition}[Poisson--Dirichlet cascade] The Poisson--Dirichlet cascade associated to $(\zeta_k)_{1\leq K+1}$ in \eqref{e.zetas=} is the random probability measure on $\N^K$ (the leaves of the tree $\mcl A$) whose weights are given by $(v_\alpha)_{\alpha \in \N^K}$.
\end{definition}
Let $(v_\alpha)_{\alpha \in \N^K}$ be the Poisson--Dirichlet cascade associated to $(\zeta_k)_{1\leq K+1}$ in \eqref{e.zetas=}, chosen to be independent of $H_N$. Let $(z_\beta)_{\beta \in \mcl A}$ be a family of independent $\R^{D \times N}$-valued Gaussian vectors with independent standard Gaussian entries. We choose $(z_\beta)_{\beta \in \mcl A}$ independent of $(v_\alpha)_{\alpha \in \N^K}$ and $H_N$. For every $\alpha \in \N^K$, we set 
\begin{e*}
    w^q(\alpha) = \sum_{k = 0}^K (q_k - q_{k-1})^{1/2} z_{\alpha_{|k}}
\end{e*}
with $(q_k)_{0\leq k\leq K}$ given in~\eqref{e.qs=}.
The centered Gaussian process $(w^q(\alpha))_{\alpha \in \N^K}$ is $\R^{D \times N}$-valued and has the following covariance structure 
\begin{e*}
    \E \left[ w^q(\alpha)w^q(\alpha')^* \right] = Nq_{\alpha \wedge \alpha'}.
\end{e*}
Henceforth, we write $\R_+=[0,+\infty)$. For $t\in\R_+$ and $q$ given in~\eqref{e.piecewise}, we define the enriched Hamiltonian
\begin{e*} H_N^{t,q}(\sigma,\alpha) = \sqrt{2t} H_N(\sigma) - Nt \xi \left( \frac{\sigma \sigma^*}{N}\right) + \sqrt{2}w^q(\alpha) \cdot \sigma - \sigma \cdot q_K \sigma,
\end{e*}
where $H_N$ is given as in~\eqref{e.H_N(sigma)=}.
And we define
\begin{e} \label{e.enriched free energy}
    F_N(t,q) = -\frac{1}{N} \E \log \int \sum_{\alpha \in \N^K} \exp \left(H_N^{t,q}(\sigma,\alpha) \right) v_\alpha \d P_N(\sigma).
\end{e}
Recall $\mcl Q$ from~\eqref{e.Q=} and $\mcl Q_p=\mcl Q\cap L^p([0,1);S^D)$.
The expression in the previous display is Lipschitz with respect to $(t,q)$. More precisely for every $t_1,t_2 \in \R_+$ every piecewise constant $q_1,q_2 \in \mcl Q$, as proven in \cite[Proposition~3.1]{chenmourrat2023cavity}, we have 
\begin{e}\label{e.F_L^1}
    |F_N(t_1,q_1) - F_N(t_2,q_2)| \leq |q_1-q_2|_{L^1} +|t_1-t_2| \sup_{|a| \leq 1} |\xi(a)|.  
\end{e}
As a consequence, the free energy admits a unique Lipschitz extension to $\R_+ \times \mcl Q_1$. The relevance of the enriched free energy is encapsulated in Theorem~\ref{t.gateaux} below, which is extracted from \cite[Corollary~8.7]{chenmourrat2023cavity}. In other words, this result states that the enriched free energy is the unique solution of some partial differential equation. This approach was explored in \cite{guerra2001sum}, in the replica-symmetric regime, and was later extended to various settings~\cite{abarra,barra2,barra1,barra2014quantum,barramulti,HJbook,mourrat2020nonconvex,mourrat2019parisi,mourrat2020free,mourrat2020extending}.

Recall the notion of Gateaux differentiability in Definition~\ref{d.gateaux diff}, the definition of $\mcl Q_{\infty,\uparrow}$ in~\eqref{e.Q_uparrow=} and~\eqref{e.Q_inftyuparrow=}, and lastly the limit free energy $f$ in~\eqref{e.f=} (which exists when $\xi$ is convex on $S^D_+$).

Given a Gateaux differentiable functions $h : \mcl Q_{\infty,\uparrow} \subset L^2 \to \R$, we have for $q \in \mcl Q_{\infty,\uparrow}$, $\nabla h(q) \in L^2$. In particular, provided that the integral converges, the following quantity is well-defined $\int_0^1 \xi(\nabla h(q)(u)) \d u$, where $u\in[0,1)$ is the variable making $\nabla h(q)$ an element of $L^2=L^2([0,1))$. Hence, we write for short
\begin{align}\label{e.meaning_nonlinearity}
    \int \xi(\nabla h(q)) = \int_0^1 \xi(\nabla h(q)(u)) \d u.
\end{align}
\begin{theorem}[The free energy solves a PDE \cite{chenmourrat2023cavity}]\label{t.gateaux}
Assume that $\xi$ is convex on $S^D_+$. The limit free energy $f$ is Gateaux differentiable at every $(t,q) \in (0,+\infty) \times \mcl Q_{\infty,\uparrow}$ and satisfies
\begin{e} \label{e.gateaux hj}
    \partial_t f - \int \xi(\nabla_q f) = 0 
\end{e}
everywhere on $(0,+\infty) \times \mcl Q_{\infty,\uparrow}$. Moreover, we have
\begin{align}\label{e.nablaf(t,q)>0}
    \nabla_q f(t,q) \in \mcl Q\cap L^\infty_{\leq 1}\subset \mcl Q_2,\qquad\forall (t,q)\in(0,+\infty)\times \mcl Q_{\infty,\uparrow}.
\end{align}
\end{theorem}
The quantity $\int \xi(\nabla_q f)$ is understood as in~\eqref{e.meaning_nonlinearity} and $L^\infty_{\leq 1}$ is defined in~\eqref{e.L^infty_<1}. 

Recall $\psi$ from~\eqref{e.psi=} and thus we have
\begin{align}\label{e.f(0,)=psi}
    f(0,\cdot) =\psi,\qquad \text{on $\mcl Q_2\supset \mcl Q_{\infty,\uparrow}$}.
\end{align}
According to \cite[Theorem~1.1]{chen2022hamilton}, the Cauchy problem \eqref{e.gateaux hj} on $(0,+\infty) \times \mcl Q_{\infty,\uparrow}$ with initial condition $f(0,\cdot)=\psi$ admits a unique solution (in the viscosity sense). Therefore, Theorem~\ref{t.gateaux} can be seen as a characterization of the limit free energy. A variational formula for $f$ analogous to the Parisi formula follows from the Hopf--Lax representation \cite[Theorem~1.1~(2)]{chen2022hamilton}. It has been conjectured in \cite{mourrat2019parisi} that Theorem~\ref{t.gateaux} is the right way of characterizing the limit free energy for models with nonconvex $\xi$. In particular, the limit of the usual free energy $\bar F_N(t,0)$ (without enrichment) as in~\eqref{e.free energy} is conjectured to be the value at $(t,0)$ of the unique solution of \eqref{e.gateaux hj} regardless of the convexity of $\xi$ \cite[Conjecture~2.6]{mourrat2019parisi}. Partial results toward a proof of this conjecture have been obtained in \cite{mourrat2020nonconvex, mourrat2020free,chenmourrat2023cavity,issa2024permutation}. Because of this, we believe that the enriched free energy is an interesting object for understanding both convex and nonconvex spin glasses and this is why we choose to study it in the remainder of this paper.
\section{Modified Hopf--Lax representation of the limit free energy} \label{ss.hopflax}
Henceforth, we always assume that $\xi$ is convex on $S^D_+$ and we let $f = \lim_{N \to +\infty} F_N$ be the pointwise limit on $\R_+ \times \mcl Q_2$ of the enriched free energy (as in~\eqref{e.f=}). Recall that $f$ is jointly Gateaux differentiable on $(0,+\infty) \times \mcl Q_{\infty,\uparrow}$. We also simplify our notation by writing $\nabla f = \nabla_q f$ for the derivative in the second coordinate. 

Our proof of Theorem~\ref{t.uniqueness} uses the differentiability of the limit free energy to apply the envelope theorem. At the heuristic level, this theorem encapsulates the following phenomenon. Consider $h : \R \to \R$ and $k : \R \times \R \to \R$ such that 
\begin{e} \label{e.envelope}
    h(x) = \sup_{x' \in \R} k(x,x').
\end{e}
Let $x \in \R$ and let $x' \in \R$ be a maximizer in \eqref{e.envelope}. Assume that $h$ is differentiable at $x$ and $k$ is differentiable at $(x,x')$. Then formally we have,
\begin{e*}
    \frac{\d }{\d x} h(x) = \partial_x k(x,x') + \partial_{x'} k(x,x') \frac{\d }{\d x} x',
\end{e*}
Since $x'$ maximizes $k(x,\cdot)$ we have $\partial_{x'} k(x,x') = 0$ and we are left with 
\begin{e*}
    \frac{\d }{\d x} h(x) = \partial_x k(x,x').
\end{e*}
When $\partial_x k(x,\cdot)$ is injective, this computation shows that there is a unique optimizer $x'$. A rigorous proof of the envelope theorem for functions defined on $\R^d$ can be found in \cite[Theorem~2.21]{HJbook}. 

To apply the envelope theorem to $f(t,q)$, we will not use the Parisi formula \eqref{e.parisi} but a  Hopf--Lax type formula (see \eqref{e.Hopf--Lax} below). A closely related variational formula has already been derived in \cite[Corollary~8.2]{chenmourrat2023cavity}, but it is different in nature. In \cite[Corollary~8.2]{chenmourrat2023cavity} the dependence in $q$ of the formula is through $\psi$ while in \eqref{e.Hopf--Lax} it is trough $\xi^*$. 

\begin{theorem}[Hopf--Lax representation] \label{t.Hopf--Lax}
Assuming that $\xi$ is strictly convex on $S^D_+$ and is superlinear on $S^D_+$, we have that, for every $(t,q) \in (0,+\infty) \times \mcl Q_{\infty,\uparrow}$, 
\begin{e} \label{e.Hopf--Lax}
        f(t,q) = \sup_{q' \in \mcl Q_\infty} \left\{ \psi(q') - t\int \xi^* \left( \frac{q'-q}{t}\right) \right\}.
    \end{e}
\end{theorem}
Here for $\phi : S^D \to \R$ and $\kappa \in L^2([0,1);S^D)$ we write for short 
\begin{e*}
    \int \phi(\kappa) = \int_0^1 \phi(\kappa(u))\d u.
\end{e*}

Before starting the proof of Theorem~\ref{t.Hopf--Lax}, we will adapt two classical results from convex duality to our context in Lemmas~\ref{l.theta} and~\ref{l.loclip} below. Recall the definition of $\theta$ in~\eqref{e.theta=}.
\begin{lemma} \label{l.theta}
    Assume that $\xi$ is convex on $S^D_+$, then we have
\begin{e*}
        \theta(x) = \xi^*(\nabla \xi(x)),\qquad\forall x \in S^D_+.
    \end{e*}
\end{lemma}
\begin{proof}
For every $x\in S^D_+$, it follows from the definition of $\xi^*$ in~\eqref{e.def.xi.star} that 
\begin{align*}
    \theta(x) &=  \nabla \xi(x) \cdot x - \xi(x) 
    \\
    &\leq \sup_{x' \in S^D_+} \left\{ \nabla \xi(x) \cdot x' - \xi(x') \right\} 
    \\
    &= \xi^*(\nabla\xi(x)).
\end{align*}
Conversely, for any $x'\in S^D_+$, the convexity of $\xi$ on $S^D_+$ implies that for every $\lambda \in (0,1]$, 
\begin{equation*}  
\frac{\xi(\lambda x'+ (1-\lambda)x) - \xi(x)}{\lambda} \leq \xi(x') - \xi(x).
\end{equation*}
Taking $\lambda \to 0$, we get 
\begin{equation*}  
    \nabla\xi(x) \cdot (x'-x)  \leq \xi(x')-\xi(x).
\end{equation*}
Rearranging, we obtain
\begin{equation*}  
     \nabla\xi(x) \cdot x' - \xi(x') \leq \nabla \xi(x) \cdot x -\xi(x).
\end{equation*}
Taking the supremum over $x' \in S^D_+$, we get $\xi^*(\nabla \xi(x)) \leq \theta(x)$.
\end{proof}
\begin{lemma} \label{l.loclip}
Assume that $\xi$ is convex on $S^D_+$ and is superlinear on $S^D_+$, then $\xi^*$ is locally Lipschitz on $\R^{D \times D}$.
\end{lemma}
\begin{proof}
For every $y \in \R^{D \times D}$, using the superlinearity of $\xi$ as in~\eqref{e.superlinear}, we have for $x \in S^D_+$ with $|x|$ large enough    
\begin{e*}
    x \cdot y - \xi(x) \leq |x|(|y| - \xi(x)/|x|) \leq 0.
\end{e*}
Thus, for $R > 0$ large enough, 
\begin{e*}
    \xi^*(y) = \sup_{x \in S^D_+,\ |x| \leq R} \left\{ x \cdot y - \xi(x) \right\} < +\infty
\end{e*}
Hence $\xi^*$ is finite on $\R^{D \times D}$. Since $\xi^*$ is also convex, we conclude from a classical result~\cite[Theorem~10.4]{rockafellar1970convex} that $\xi^*$ is locally Lipschitz.
\end{proof}
\begin{proof}[Proof of Theorem~\ref{t.Hopf--Lax}]
It is clear from Theorem~\ref{t.parisi} and Lemma~\ref{l.theta} that
\begin{e} \label{e.f(t,q)<sup}
    f(t,q) \leq \sup_{q' \in \mcl Q_\infty} \left\{ \psi(q') - t\int \xi^* \left( \frac{q'-q}{t}\right) \right\}.
\end{e}
Let us prove the converse bound. We fix any $q' \in \mcl Q_{\infty,\uparrow}$ and set 
\begin{e*}
    {\gamma}_s = q' + \frac{s}{t}(q-q'),\quad\forall s\in[0,t].
\end{e*}
From the definition of $\mcl Q_{\infty,\uparrow}$ in~\eqref{e.Q_inftyuparrow=}, we have ${\gamma}_s \in \mcl Q_{\infty,\uparrow}$ for every $s\in[0,t]$. According to Theorem~\ref{t.gateaux}, $f$ is jointly Gateaux differentiable on $(0,+\infty) \times \mcl Q_{\infty,\uparrow}$. Hence, for every $s \in (0,t)$,
\begin{align*}
    \frac{\d}{\d s} f(s,{\gamma}_s) &= \partial_t f(s,{\gamma}_s) + \langle \dot{\gamma}_s , \nabla_q f(s,{\gamma}_s) \rangle_{L^2}\\
    &\stackrel{\eqref{e.gateaux hj}}{=} \int_0^1 \xi(\nabla f(s,{\gamma}_s))(u) \d u + \langle \dot{\gamma}_s , \nabla f(s,{\gamma}_s) \rangle_{L^2}.
\end{align*}
By definition of $\xi^*$ in~\eqref{e.def.xi.star} and $\nabla f(s,{\gamma}_s) \in \mcl Q_2$ due to~\eqref{e.nablaf(t,q)>0}, we have for every $u \in [0,1)$, 
\begin{e*}
    \xi^*(-\dot{\gamma}_s(u)) \geq -\dot{\gamma}_s(u) \cdot \nabla f(s,{\gamma}_s)(u) - \xi(\nabla f(s,{\gamma}_s)(u))
\end{e*}
which along with the previous display implies
\begin{align*}
    \frac{\d}{\d s} f(s,{\gamma}_s) \geq  -\int_0^1 \xi^*(-\dot{\gamma}_s(u)) \d u  = -\int_0^1 \xi^* \left( \frac{q'(u)-q(u)}{t}\right) \d u.
\end{align*}
Integrating this with respect to $s$ and using $f(0,\gamma_0) = \psi(q')$ due to~\eqref{e.f(0,)=psi}, we obtain,
\begin{e*}
    f(t,q) \geq \psi(q') - t\int_0^1 \xi^* \left( \frac{q'(u)-q(u)}{t}\right) \d u.
\end{e*}
Since $q'$ is arbitrary, we get 
\begin{e*}
    f(t,q) \geq \sup_{q' \in \mcl Q_{\infty,\uparrow} } \left\{ \psi(q') - t\int_0^1 \xi^* \left( \frac{q'(u)-q(u)}{t}\right) \d u \right\}.
\end{e*}
To conclude, we observe that the set 
\begin{e*}
    \left\{ u \mapsto uc \text{Id} + \sum_{k = 1}^K q_k \1_{[\frac{k-1}{K},\frac{k}{K})}\  \Big| \ c > 0;\; q_1=0; \; q_2,\dots q_k \in S^D_{++} \right\}
\end{e*}
is dense in $\mcl Q_\infty$ with respect to $L^1$-convergence and is contained in $\mcl Q_{\infty,\uparrow}$. Since $\xi^*$ is locally Lipschitz as a consequence of Lemma~\ref{l.loclip} and $\psi$ is continuous in $L^1$ due to~\eqref{e.psi=} and~\eqref{e.F_L^1}, by a density argument, we obtain 
\begin{e*}
    f(t,q) \geq \sup_{q' \in \mcl Q_\infty} \left\{ \psi(q') - t\int_0^1 \xi^* \left( \frac{q'(u)-q(u)}{t}\right) \d u \right\}.
\end{e*}
This together with~\eqref{e.f(t,q)<sup} completes the proof.
\end{proof}
\section{Uniqueness of Parisi measures for vector spins} \label{s.uniqueness}

As discussed in detail at the beginning of Section~\ref{ss.hopflax}, to prove Theorem~\ref{t.uniqueness} we use the envelope theorem. To do so, we will again need to adapt a classical result from convex duality for the differentiability of $\xi^*$ (see Lemma~\ref{l. diff xi^*}) taking into account that the supremum in the definition of $\xi^*$ in~\eqref{e.def.xi.star} is taken over $S^D_+$ instead of the whole space. We will also need to show that even though $\mcl Q_2$ has an empty interior, $\mcl Q_{\infty,\uparrow}$ satisfies some sort of openness condition. More precisely, we are going to show that given $q \in \mcl Q_{\infty,\uparrow}$ and a Lipschitz function $\kappa : [0,1) \to S^D$, one has $q +\varepsilon \kappa \in \mcl Q_{\infty,\uparrow}$ for every $\epsilon$ small enough (see Lemma~\ref{l.openness}).  

\begin{lemma} \label{l. diff xi^*}
    Assume that $\xi$ is strictly convex on $S^D_+$ and is superlinear on $S^D_+$. Then, $\xi^*$ is continuously differentiable on $S^D$. In addition, we have 
\begin{e*} \nabla \xi^*(\nabla \xi(x)) = x,\quad\forall x \in S^D_+.
    \end{e*}
\end{lemma}
\begin{proof}
    Since $\xi^*$ is convex, it is enough (e.g.\ ~\cite[Corollary~25.5.1]{rockafellar1970convex}) to show that $\xi^*$ is differentiable on $S^D$ to prove that $\xi^*$ is continuously differentiable. We fix any $y \in S^D$ and want to show that $\xi^*$ is differentiable at $y$. We consider the subdifferential 
\begin{e}\label{e.partialxi^*(y)=}
        \partial \xi^*(y) = \Ll\{ x \in S^D \ \big| \ \forall y' \in S^D, \, \xi^*(y')  - \xi^*(y) \geq x \cdot(y'-y) \Rr\}.
    \end{e}
According to Lemma~\ref{l.loclip}, $\xi^*$ is finite at $y$. Then, it follows from \cite[Proposition~5.3]{ET}, if $\partial \xi^*(y) = \{x\}$ for some $x \in S^D$, then $\xi^*$ is differentiable at $y$ with $\nabla \xi^*(y) =x$. We are going to show that $\partial \xi^*(y)$ is contained in the set of maximizers of the strictly concave functional, $x \mapsto x \cdot y - \xi(x)$.

    \noindent Step 1. We show that $\partial \xi^*(y) \subset S^D_+$.

    \noindent Recall (e.g.\ \cite[Theorem 7.5.4]{horn2012matrix} that given $x \in S^D$, we have $x \in S^D_+$ if and only if
\begin{e*}
        x \cdot a \geq 0,\quad\forall a \in S^D_+.
    \end{e*}
Let $x \in \partial \xi^*(y)$, for every $a \in S^D_+$ we let $y' = y - a$. We have 
\begin{align*}
        \xi^*(y') &= \sup_{x' \in S^D_+} \left\{ x' \cdot y' - \xi(x') \right\} \\
                 &\leq\sup_{x' \in S^D_+} \left\{ x' \cdot y - \xi(x') \right\} \\
                 &= \xi^*(y).
    \end{align*}
Thus, 
\begin{e*}
        x \cdot a = x \cdot  (y-y') \stackrel{\eqref{e.partialxi^*(y)=}}{\geq} \xi^*(y) - \xi^*(y') \geq 0,
    \end{e*}
which proves that $x \in S^D_+$.
    
    \noindent Step 2. We show that $\xi^*(y) \leq x \cdot y - \xi(x)$ for every $x \in \partial \xi^*(y)$ and that this implies that $\xi^*$ is differentiable at $y$.
    
    \noindent Fix any $x\in \partial \xi^*(y)$. For every $y' \in S^D_+$, we have 
\begin{e*}
        x \cdot y - \xi^*(y) \stackrel{\eqref{e.partialxi^*(y)=}}{\geq} x \cdot y' - \xi^*(y').
    \end{e*}
According to Step 1, we have $x \in S^D_+$, so by the biconjugation theorem on $S^D_+$ \cite[Theorem~2.2]{chen2022fenchelmoreau} (for which we need the convexity of $\xi$ and the monotonicity of $\xi$ in~\eqref{e.xi_monotonicity}) we have \begin{e*}
        \xi(x) = \xi^{**}(x)=\sup_{y' \in S^D_+} \left\{ x \cdot y' - \xi^*(y') \right\}.
    \end{e*}
Therefore, from the previous two displays, we obtain 
\begin{e*}
         x \cdot y - \xi^*(y) \geq \xi(x),
    \end{e*}
which is the desired inequality. As a consequence, we have $x=x_0$ where $x_0$ is the unique maximizer of the strictly concave function $x' \mapsto x' \cdot y - \xi(x)$ over $S^D_+$. Since $x\in\partial\xi^*(y)$ is arbitrary, we get $\partial \xi^*(y) = \{ x_0\}$. Hence, as explained in the beginning, we conclude that $\xi^*$ is differentiable at $y$ with $\nabla \xi^*(y) = x_0$.

    \noindent Step 3. We show that for every $x \in S^D_+$, $\nabla \xi^*(\nabla \xi(x)) = x$.

    \noindent From Lemma~\ref{l.theta}, we know that
\begin{e*}
        \xi^*( \nabla \xi(x)) = x \cdot \nabla \xi(x) - \xi(x).
    \end{e*}
Thus, $x$ is the maximizer of the strictly concave functional $x' \mapsto x' \cdot \nabla \xi(x) - \xi(x')$ over $S^D_+$. Using the last part of Step~2 with $y$ substituted with $\nabla\xi(x)$, we get $\nabla \xi^*(\nabla \xi(x)) = x$, which completes the proof.
\end{proof}

Recall the definition of $\ellipt(a)$ in~\eqref{e.Ellipt(a)=} and those of $\lambda_{\max}(a)$ and $\lambda_{\min}(a)$ above~\eqref{e.Ellipt(a)=}.
It follows from Weyl's inequalities that $\lambda_{\max}$ and $\lambda_{\min}$ are respectively sub-additive and super-additive. Therefore, given two symmetric matrices $a,b \in S^D$ such that $a+b \in S^D_{++}$, $\lambda_{\max}(a) + \lambda_{\max}(b) \geq 0$ and $\lambda_{\min}(a) + \lambda_{\min}(b) > 0$, we have
\begin{e} \label{e.ellipt bounds}
    \ellipt(a+b) \leq \frac{\lambda_{\max}(a) + \lambda_{\max}(b)}{\lambda_{\min}(a) + \lambda_{\min}(b)}.
\end{e}

\begin{lemma} \label{l.openness}
     Let $\kappa : [0,1) \to S^D$ be a Lipschitz function such that $\kappa(0) = 0$ and $q \in \mcl Q_{\infty,\uparrow}$. We have $q +\varepsilon\kappa \in \mcl Q_{\infty,\uparrow}$ for $\varepsilon > 0$ small enough. 
\end{lemma}

\begin{proof}
Recall that we write $a\geq b$ if $a,b\in S^D$ satisfies $a-b\in S^D_+$.
Since $\kappa$ is Lipschitz, there is a constant $L>0$ such that
\begin{align}\label{e.kappa(v)-kappa(u)<}
    \kappa(v) - \kappa(u) \leq L|u-v|\Id,\quad\forall v,\,u \in [0,1).
\end{align}
This implies that, for every $v, u \in [0,1)$, we have
\begin{align}\label{e.lambda(kappa(v)-kappa(u))}
    \lambda_{\max}(\kappa(v)-\kappa(u)),\ \lambda_{\min}(\kappa(v)-\kappa(u)) \in [-L(v-u), L(v-u)].
\end{align}
The definition of $\mcl Q_{\infty,\uparrow}$ in~\eqref{e.Q_inftyuparrow=} and~\eqref{e.Q_uparrow=} gives a constant $c>0$ such that
\begin{e}\label{e.q(v)-q(u)}
    q(v) -q(u) \geq c(v-u) \text{Id} \quad\text{and}\quad \ellipt(q(v)-q(u)) \leq c^{-1}.
\end{e}
Fixing $u <v$, we have 
\begin{align*}
    (q(v)+\varepsilon \kappa(v))-(q(u)+\varepsilon \kappa(u)) &\stackrel{\eqref{e.kappa(v)-kappa(u)<}\eqref{e.q(v)-q(u)}}{\geq} c(v-u)\Id - L\varepsilon(v-u)\Id \\
    &= (c -L\varepsilon)(v-u)\Id.
\end{align*}
We also have
\begin{align*}
    &\ellipt((q(v)+\varepsilon \kappa(v))-(q(u)+\varepsilon \kappa(u))) 
    \\
    &\stackrel{\eqref{e.ellipt bounds}\eqref{e.lambda(kappa(v)-kappa(u))}}{\leq} \frac{\lambda_{\max}(q(v) - q(u)) + L\varepsilon|u-v|}{\lambda_{\min}(q(v) - q(u)) - L\varepsilon |u-v|} \\
    &\stackrel{\eqref{e.q(v)-q(u)}}{\leq} \frac{c^{-1}(\lambda_{\min}(q(v) - q(u))-L\varepsilon|u-v|) + (c^{-1}+1)L\varepsilon|u-v|}{\lambda_{\min}(q(v) - q(u)) - L\varepsilon |u-v|} \\
    &\leq c^{-1} + \frac{(c^{-1}+1)L\varepsilon|u-v|}{\lambda_{\min}(q(v) - q(u)) - L\varepsilon |u-v|} \\
    &\stackrel{\eqref{e.q(v)-q(u)}}{\leq} c^{-1} + \frac{(c^{-1}+1)L\varepsilon|u-v|}{c|u-v| - L\varepsilon |u-v|} \\
    &\leq c^{-1} +\frac{(c^{-1}+1)L\varepsilon}{c - L\varepsilon}.
\end{align*}
Furthermore, since $q \in L^\infty$ and $\kappa$ is Lipschitz, it is clear that $q +\varepsilon \kappa$ is bounded. Since $\kappa(0) = 0$ it is also clear that $(q +\varepsilon \kappa)(0)= 0$. Now, comparing these and the above two displays with the definitions in~\eqref{e.Q_uparrow=} and~\eqref{e.Q_inftyuparrow=}, we conclude that $q +\varepsilon \kappa \in \mcl Q_{\infty,\upa}$ for sufficiently small $\eps$.
\end{proof}

\begin{proof}[Proof of Theorem~\ref{t.uniqueness}]
For brevity, we write
\begin{e*}
    \msc H_t(\rho,\rho') = \psi(\rho') - t \int \xi^* \left( \frac{\rho'-\rho}{t}\right),\qquad\forall \rho,\rho' \in \mcl Q_\infty.
\end{e*}
Using this notation, we can rewrite the Hopf--Lax representation~\eqref{e.Hopf--Lax} of the Parisi formula in Theorem~\ref{t.Hopf--Lax} as
\begin{align}\label{e.parisi_rewrite}
    f(t,q) = \sup_{q'\in\mcl Q_\infty}\msc H_t\Ll(q,q'\Rr).
\end{align}
We fix $(t,q) \in (0,+\infty) \times \mcl Q_{\infty,\uparrow}$. 
Let $p \in \mcl Q_\infty$ be a Parisi measure (see Definition~\ref{d.parisi_measure}) at $(t,q)$. Set $q' = q + t \nabla \xi(p)$. 
Using Lemma~\ref{l.theta} to rewrite $\theta$ in~\eqref{e.parisi},
we have $f(t,q) = \msc H_t(q,q')$. So, $q'$ is a maximizer in the right-hand side of \eqref{e.parisi_rewrite}. Let $\kappa : [0,1) \to S^D$ be a Lipschitz function with $\kappa(0) = 0$, we have
\begin{e*}
    \frac{\msc H_t(q+\varepsilon \kappa,q') -\msc H_t(q,q')}{\varepsilon} = -\frac{1}{\varepsilon} \int t\xi^* \left( \frac{q'-q - \varepsilon \kappa }{t}\right)- t\xi^* \left( \frac{q'-q}{t}\right).
\end{e*}
The integral on the right-hand side is $\int = \int_0^1\d u$ over the parameter of paths.
According to Lemma~\ref{l. diff xi^*}, the function $\xi^*$ is continuously differentiable on $S^D$. In particular, $\xi^*$ is locally Lipschitz and we can apply the dominated convergence theorem in the previous display to discover that 
\begin{e*}
    \lim_{\varepsilon \downarrow 0} \frac{\msc H_t(q+\varepsilon \kappa,q') -\msc H_t(q,q')}{\varepsilon} =  \left\langle \nabla \xi^* \left( \frac{q'-q}{t}\right) , \kappa \right\rangle_{L^2}.
\end{e*}
By Lemma~\ref{l.openness}, we have $q +\varepsilon \kappa \in \mcl Q_{\infty,\uparrow}$ for every $\varepsilon > 0$ small enough. Due to~\eqref{e.parisi_rewrite}, we have $f(t,q+\varepsilon \kappa) \geq \msc H_t(q+\varepsilon \kappa,q')$ and thus
\begin{e*}
    \frac{f(t,q+\varepsilon \kappa) -f(t,q)}{\varepsilon} \geq \frac{\msc H_t(q+\varepsilon \kappa,q') -\msc H_t(q,q')}{\varepsilon}.
\end{e*}
By Theorem~\ref{t.gateaux}, $f(t,\cdot)$ is Gateaux differentiable (see Definition~\ref{d.gateaux diff}) at $q$. Letting $\varepsilon \to 0$ in the previous display, we obtain
\begin{e*}
   \langle  \nabla f(t,q) , \kappa \rangle_{L^2}  \geq \Ll\langle \nabla \xi^* \left( \frac{q'-q}{t}\right), \kappa \Rr\rangle_{L^2}.
\end{e*}
Since $\kappa$ is an arbitrary Lipschitz function vanishing at $u = 0$ and those functions are dense in $L^2$, this yields $\nabla f(t,q) = \nabla \xi^* \left( \frac{q'-q}{t}\right)$. In addition, according to Lemma~\ref{l. diff xi^*}, we have 
\begin{e*}
    \nabla \xi^* \left( \frac{q'-q}{t}\right) = \nabla \xi^* (\nabla \xi(p)) = p.
\end{e*}
So in conclusion, $p = \nabla f(t,q)$.
\end{proof}

\section{Fréchet differentiability of the limit free energy} \label{ss.frechet}

Finally, as an application of Theorem~\ref{t.uniqueness} we will upgrade the Gateaux differentiability result of \cite[Proposition~8.1]{chenmourrat2023cavity} to a Fréchet differentiability result.

\begin{definition}[Fréchet differentiability]\label{d.frechet}
Let $(E,\,|\,\cdot\,|_E)$ be a Banach space and denote by $\langle \,\cdot\,,\, \cdot\, \rangle_E$ the canonical pairing between $E$ and $E^*$. Let $G$ be a subset of $E$ and $q \in G$, a function $h : G \to \R$ is Fréchet differentiable at $q \in G$ if there exists a unique $y^* \in E^*$ such that the following holds:
\begin{e*}
        \lim_{\substack{q' \to q \\ q' \in G}} \frac{h(q') - h(q) - \langle y^*, q'-q \rangle_E}{|q'-q|_E} = 0.
    \end{e*}
In such a case, we call $y^*$ the Fréchet derivative of $h$ at $q$ and we denote it $\nabla h(q)$.
\end{definition}
When a function $h : G \to \R$ is Fréchet differentiable at every, $q \in G$ we say that $h$ is Fréchet differentiable everywhere on $G$. 
\begin{theorem}[Fréchet differentiability of the free energy] \label{t.frechet}
Assume that $\xi$ is convex on $S^D_+$ and is superlinear on $S^D_+$. The limit free energy $f$ as in~\eqref{e.f=} is Fréchet differentiable everywhere on $(0,+\infty) \times \mcl Q_{\infty,\uparrow}$. 
\end{theorem}
Recall that from Theorem~\ref{t.gateaux} which is extracted from \cite{chenmourrat2023cavity} that the Gateaux derivative of $f$ satisfy 
\begin{align*}
    \partial_t f - \int \xi(\nabla_q f) = 0,
\end{align*}
everywhere on $(0,+\infty) \times \mcl Q_{\infty,\uparrow}$. Since the Gateaux and Fréchet derivatives coincide when they exist, this partial differential equation can be understood in the Fréchet sense as well.
\begin{proof}[Proof of Theorem~\ref{t.frechet}]
In this proof, for clarity, we fix $t \in (0,+\infty)$ and only show that $f(t,\cdot)$ is Fréchet differentiable at every $q \in \mcl Q_{\infty,\uparrow}$. The proof of the joint differentiability is the same but with more cumbersome notations, which we choose to omit here. For $\rho \in\mcl Q_2$ and $\rho'\in\mcl Q_\infty$, we write 
\begin{e*}
    \msc P_t(\rho,\rho') = \psi(\rho+ t \nabla \xi(\rho')) - t \int \theta(\rho').
\end{e*}
Before proceeding, we record some properties of $\msc P_t$. 
According to \cite[Corollary~5.2]{chenmourrat2023cavity}, $\psi$ is Fréchet differentiable on $\mcl Q_2$. Therefore, $\msc P_t(\cdot,\rho')$ is Fréchet differentiable at every $\rho \in \mcl Q_2$. We denote by $\nabla_\rho \msc P_t(\rho,\rho')$ its derivative at $\rho$, which has the expression:
\begin{e*}
    \nabla_\rho \msc P_t(\rho,\rho') = \nabla \psi(\rho+ t \nabla \xi(\rho')).
\end{e*}
Also, \cite[Corollary~5.2]{chenmourrat2023cavity} gives that $\nabla\psi$ is Lipschitz in $L^2$. Using this and the smoothness of $\xi$, there is a constant $C>0$ such that
\begin{align}\label{e.|nablaP-nablaP|<}
    \Ll|\nabla_\rho \msc P_t(\rho_1,\rho_1') - \nabla_\rho \msc P_t(\rho_2,\rho_2')\Rr|\leq C\Ll|\rho_1-\rho_2\Rr|_{L^2}+ Ct \Ll|\rho'_1-\rho'_2\Rr|_{L^2}
\end{align}
for every $\rho_1,\rho_2\in\mcl Q_2$ and $\rho'_1,\rho'_2\in \mcl Q\cap L^\infty_{\leq1}$. Notice that the boundedness on $\rho'_1$ and $\rho'_2$ is needed here.

According to Theorem~\ref{t.parisi}, we have 
\begin{e*}
    f(t,q) = \sup_{p \in \mcl Q_\infty} \msc P_t(q,p).
\end{e*}
Let $(q_n)_n$ be any sequence in $\mcl Q_{\infty,\uparrow} \setminus \{ q \}$ such that $q_n \to q$ in $L^2$. Let $p_n \in \mcl Q_\infty$ be the unique Parisi measure at $(t,q_n)$ given by Theorem~\ref{t.uniqueness}. The theorem also gives $p_n = \nabla f(t,q_n) \in \mcl Q \cap L^\infty_{\leq 1}$. 
By~\cite[Lemma~3.4]{chenmourrat2023cavity}, thanks to the monotonicity and boundedness of $p_n$ uniform in $n$, the sequence $(p_n)_n$ is pre-compact in $L^2$.
Let $p$ be a subsequential limit of $(p_n)_n$. Clearly, we still have $p\in \mcl Q\cap L^\infty_{\leq1}$. Passing to the limit in $f(t,q_n) = \msc P_t(q_n,p_n)$, we discover that $p$ is a Parisi measure at $(t,q)$. Since Theorem~\ref{t.uniqueness} ensures that such $p$ is unique, we deduce that the sequence $(p_n)$ converges in $L^2$ to the unique Parisi measure $p$ at $(t,q)$. It follows that 
\begin{align*}
    f(t,q_n) - f(t,q) &\leq \msc P_t(q_n,p_n) - \msc P_t(q,p_n) \\
    &= \int_0^1 \langle q_n-q, \nabla_\rho \msc P_t( \lambda q_n +(1-\lambda)q ,p_n) \rangle_{L^2} \d \lambda.
\end{align*}
We thus have
\begin{e*}
    \begin{split}
        f(t,q_n) - f(t,&q) - \langle q_n-q, \nabla_\rho \msc P_t(q ,p) \rangle_{L^2} \\ &\leq  \int_0^1 \langle q_n-q, \nabla_\rho \msc P_t( \lambda q_n +(1-\lambda)q ,p_n) - \nabla_\rho \msc P_t(q ,p) \rangle_{L^2} \d \lambda.
    \end{split}
\end{e*}
Now observe that
\begin{e*}
    \begin{split} &\int_0^1 \langle q_n-q, \nabla_\rho \msc P_t( \lambda q_n +(1-\lambda)q ,p_n) - \nabla_\rho \msc P_t(q ,p) \rangle_{L^2} \d \lambda \\ &\leq |q_n-q|_{L^2} \int_0^1 | \nabla_\rho \msc P_t( \lambda q_n +(1-\lambda)q ,p_n) - \nabla_\rho \msc P_t(q ,p) |_{L^2} \d \lambda. \end{split}
\end{e*}
Due to~\eqref{e.|nablaP-nablaP|<}, the integral on the right-hand side vanishes as $n \to +\infty$. As a result, we get 
\begin{e*}
    f(t,q_n) - f(t,q) - \langle q_n-q, \nabla_\rho \msc  P_t(q ,p) \rangle_{L^2} \leq o(|q_n-q|_{L^2}).
\end{e*}
In addition, we also have,
\begin{e*}
\begin{split}
     f(t,q_n) - f(t,&q) - \langle q_n-q, \nabla_\rho \msc P_t(q ,p) \rangle_{L^2} \\ &\geq \msc P_t(q_n,p) -\msc P_t(q,p) - \langle q_n-q, \nabla_\rho \msc P_t(q ,p) \rangle_{L^2} = o(|q_n-q|_{L^2}).
\end{split}       
\end{e*}
In conclusion, 
\begin{e*}
    f(t,q_n) = f(t,q) + \langle q_n-q, \nabla_\rho \msc P_t(q ,p) \rangle_{L^2} + o(|q_n-q|_{L^2}).
\end{e*}
By sequential characterization of the limit, we deduce that $f$ is Fréchet differentiable at $q$.
\end{proof}

\section{Uniqueness in the critical point representation} \label{s.unique crit point}

In \cite{chenmourrat2023cavity}, a representation of the limit free energy in terms of the critical points of a functional has been proven. More precisely, we consider the functional
\begin{e}\label{e.J_t,q(q',p)=}
    \mcl J_{t,q}(q',p) = \psi(q') + \langle p, q - q' \rangle_{L^2} +t \int \xi(p).
\end{e}
We say that $(q',p) \in \mcl Q_2 \times L^2$ is a critical point of $\mcl J_{t,q}$ when
\begin{e}\label{e.crit_relation}
    \begin{cases}
        q' = q + t \nabla \xi(p) \\
        p = \nabla \psi(q')
    \end{cases}.
\end{e}
A consequence of the main result \cite[Theorem~1.2]{chenmourrat2023cavity} is that, for every $(t,q) \in \R_+ \times \mcl Q_2$ there exists $(q',p) \in \mcl Q_\infty^2$ that is a critical point of $\mcl J_{t,q}$ and such that 
\begin{e*}
    f(t,q) = \mcl J_{t,q}(q',p).
\end{e*}
Using Theorem~\ref{t.uniqueness} we can show that this critical point representation for the free energy is in fact unique as soon as $\xi$ is strictly convex on $S^D_+$.
\begin{corollary}[Uniqueness in the critical point representation]
     Assume that $\xi$ is strictly convex on $S^D_+$ and is superlinear on $S^D_+$. Let $(t,q) \in (0,+\infty) \times \mcl Q_{\infty,\uparrow}$, there exists a \emph{unique} critical point $(q',p) \in \mcl Q_\infty$ of $\mcl J_{t,q}$ such that 
\begin{e*}
        f(t,q) = \mcl J_{t,q}(q',p).
    \end{e*}
\end{corollary}
\begin{proof}
Let $(q',p) \in \mcl Q_\infty^2$ be a critical point of $\mcl J_{t,q}$. Using $q'-q  = t \nabla \xi(p)$, we get 
\begin{align*}
    \mcl J_{t,q}(q',p) 
   &\stackrel{\eqref{e.J_t,q(q',p)=}}{=} \psi(q+t \nabla \xi(p)) - t \left( \langle p, \nabla \xi(p) \rangle_{L^2} - \int \xi(p) \right) \\
   &\stackrel{\eqref{e.theta=}}{=}\psi(q+t \nabla \xi(p)) - t \int \theta(p).
\end{align*}
If we further assume that $f(t,q) = \mcl J_{t,q}(q',p)$, then we obtain 
\begin{e*}
    f(t,q)= \psi(q+t \nabla \xi(p)) - t \int \theta(p).
\end{e*}
This means that $p$ is a Parisi measure, which by Theorem~\ref{t.uniqueness} imposes that $p = \nabla f(t,q)$. By the critical point condition~\eqref{e.crit_relation}, we also have $q' = q +t \nabla \xi(\nabla f(t,q))$, which uniquely characterizes $(q',p)$.
\end{proof}

\appendix

\section{Uniqueness of Parisi measures for scalar spins} \label{s.auffingerchen}

In this section, we expand on the proof of the uniqueness of Parisi measures at $(t,0)$ in the case of scalar spins (i.e.\ $D =1$) given in \cite{auffinger2015parisi}. More precisely, using the Kantorovich duality from optimal transport, we show that this proof also yields (still in the case of scalar spins) the uniqueness of Parisi measures at $(t,q)$ for any $q \in \mcl Q_\infty$.

As written, the proof of the strict concavity of the Parisi functional given in \cite{auffinger2015parisi} assumes that $P_1$ is the uniform measure on $\{-1,1\}$. We will make this assumption here as well, but it seems that with a bit more care some more general reference measures can be considered using the results of \cite[Theorem~4.6]{chen2023parisipdevector}.  

When $D = 1$, we can obtain \eqref{e.Hopf--Lax} (as done in~\cite[Proposition~A.3]{chen2022hamilton}) using an elementary rearrangement argument instead of relying on the Gateaux differentiability of $f$ as in the proof of Theorem~\ref{t.uniqueness}. 
In this setting, it has been observed that at $q = 0$ the Parisi formula can be recast into a strictly concave optimization problem and admits a unique maximizer \cite{auffinger2015parisi}. Using~\eqref{e.Hopf--Lax}, this concavity, and the convexity of the optimal transport cost between probability measures, we show that the uniqueness of Parisi measures at every $q \in \mcl Q_\infty$. Recall the notion of Parisi measures in Definition~\ref{d.parisi_measure}.
\begin{proposition}[Uniqueness of Parisi measures for scalar spins]
    Assume that $D = 1$ and that $P_1$ is the uniform probability measure on $\{-1,1\}$. For every $(t,q) \in (0,+\infty) \times \mcl Q_\infty$ there is a unique Parisi measure at $(t,q)$. 
\end{proposition}
\begin{proof}
Let $\mcl P(\R_+)$ denote the set of probability measures on $\R_+$. We denote by $\mcl P_2(\R_+)$ (respectively, $\mcl P_\infty(\R_+)$) the set of probability measures on $\R_+$ with finite second moments (respectively, compact supports). We equip $\mcl P_2(\R_+)$ with $W_2$ the Wasserstein distance defined by
\begin{e*}
   W_2(\mu',\mu) = \left( \inf_{\pi \in \Pi(\mu',\mu)} \int |x'-x|^2 \d \pi(x',x) \right)^{1/2}
\end{e*}
where $\Pi(\mu',\mu)$ denotes the set of probability measures $\pi \in \mcl P_2(\R_+ \times \R_+)$ with first marginal $\mu'$ and second marginal $\mu$. Let $U$ be a uniform probability measure on $[0,1)$. We consider the map 
\begin{e*}
    G : \begin{cases}
        \mcl Q_2 \to \mcl P_2(\R_+) \\
        q \mapsto \law(q(U))
    \end{cases}.
\end{e*}
The functional $G$ is an isometric bijection between the metric spaces $(\mcl Q_2,|\,\cdot~\,|_{L^2})$ and $(\mcl P_2(\R_+),W_2)$. It has been observed in \cite[Theorem~2]{auffinger2015parisi} that the map $\mu \mapsto \psi(G^{-1}(\mu))$ is strictly concave. More precisely, for every $\mu_0,\mu_1 \in \mcl P_\infty(\R_+)$ and $\lambda \in [0,1]$ we have 
\begin{e*}
   \psi\Ll(G^{-1}(\lambda\mu_1 +(1-\lambda) \mu_0)\Rr) \geq \lambda \psi\Ll(G^{-1}(\mu_1)\Rr) + (1-\lambda) \psi\Ll(G^{-1}(\mu_0)\Rr) 
\end{e*}
with equality if and only if $\lambda = 0$, $\lambda = 1$, or $\mu_0 = \mu_1$.

Given $\mu',\mu \in \mcl P_2(\R_+)$, we denote by $\mcl T_t(\mu',\mu)$ the cost of the optimal transport between $\mu'$ and $\mu$ with respect to the cost function $(x',x) \mapsto t \xi^* \left( \frac{x'-x}{t}\right)$, that is, 
\begin{e*}
   \mcl T_t(\mu',\mu) =  \inf_{\pi \in \Pi(\mu',\mu)} \int t \xi^* \left( \frac{x'-x}{t}\right) \d \pi(x',x).
\end{e*}
According to \cite[Proposition~2.5]{mourrat2020free}, for every $q',q \in \mcl Q_2$ we have 
\begin{e}\label{e.T_t(G(q'),G(q))=}
   \mcl T_t(G(q'),G(q)) = \int_0^1 t \xi^* \left( \frac{q'(u)-q(u)}{t}\right) \d u.
\end{e}
In other words, the optimal coupling between measures $G(q')$ and $G(q)$ to minimize the aforementioned cost is achieved by the joint $(q'(U),q(U))$ where $U$ is the uniform random variable over $[0,1)$.

Fix any $(t,q) \in (0,+\infty) \times \mcl Q_\infty$ and set $\mu = G(q)$. The formula for $f$ obtained in \cite[Proposition~A.3]{chen2022hamilton} (which is simply \eqref{e.Hopf--Lax} but with a simpler proof) can be written as 
\begin{e*}
   f(t,q) = \sup_{\mu' \in \mcl P_\infty(\R_+)} \left\{ \psi(G^{-1}(\mu')) - \mcl T_t(\mu',\mu) \right\}.
\end{e*}

\noindent Step 1. We show that the right-hand side in the previous display is a strictly concave optimization problem. In particular, this optimization problem admits a unique maximizer.

\noindent The Kantorovich duality theorem (see for example \cite[Theorem~5.10]{villani2}) states that the optimal transport cost $\mcl T_t(\mu',\mu)$ admits the following dual representation 
\begin{e*}
   \mcl T_t(\mu',\mu) = \sup_{\chi',\chi} \left\{\int \chi' d\mu' - \int \chi d\mu \right\},
\end{e*}
where the supremum is taken over all pairs of functions $\chi,\chi' : \R_+\to \R$ such that $\chi'(x')- \chi(x) \leq (t \xi)^*(x'-x)$. When the optimal transport cost is written as in the previous display, it is the supremum of a family of linear functions in $(\mu',\mu)$. In particular, the function $(\mu',\mu) \mapsto \mcl T_t(\mu',\mu)$ is convex on $\mcl P_\infty(\R_+) \times \mcl P_\infty(\R_+)$. Therefore, since the map $\mu' \mapsto \psi(G^{-1}(\mu'))$ is strictly concave according to \cite[Theorem~2]{auffinger2015parisi}, we have that for every $\mu \in \mcl P_\infty(\R_+)$, the map $\mu' \mapsto \psi(G^{-1}(\mu')) - \mcl T_t(\mu',\mu)$ is strictly concave on $\mcl P_\infty(\R_+)$.

\noindent Step 2. We show that there is a unique Parisi measure at $(t,q)$.

\noindent Let $p$ be a Parisi measure at $(t,q)$, let $q' = q + t \nabla \xi(p)$, and let $\mu' = G(q')$. We have
\begin{align*}
   f(t,q) &\stackrel{\text{D.\ref{d.parisi_measure}}}{=} \psi(q +t \nabla \xi(p)) - t \int \theta(p) \\ 
          &\stackrel{\text{L.\ref{l.theta}}}{=} \psi(q') - \int_0^1 t \xi^* \left( \frac{q'(u)-q(u)}{t}\right) \d u \\
          &\stackrel{\eqref{e.T_t(G(q'),G(q))=}}{=} \psi(G^{-1}(\mu')) -  \mcl T_t(\mu',\mu).
\end{align*}
So, $\mu'$ is the unique maximizer in the variational formula of Step 1. In addition, according to Lemma~\ref{l. diff xi^*}, we have 
\begin{e*}
   p = \nabla \xi^*(\nabla \xi(p)) = \nabla \xi^* \left( \frac{q'-q}{t} \right) = \nabla \xi^* \left( \frac{G^{-1}(\mu')-q}{t} \right),
\end{e*}
which characterizes $p$ uniquely.
\end{proof}

Finally, we explain the difficulty that arises when one tries to generalize the arguments of \cite{auffinger2015parisi} to $D > 1$. The argument relies on the fact that through the change of variable $q \mapsto \text{Law}(q(U))$, the Parisi formula becomes a strictly concave optimization problem, and thus it admits only one maximizer. Let $\mcl P^\upa(S^D_+)$ denote the image of $\mcl Q$ by the map $p \mapsto \text{Law}(q(U))$. When $D = 1$, $\mcl P^\upa(S^D_+)$ is simply the set of probability measures on $\R_+$ which is a convex set (under the linear structure on the set of signed measures). When $D > 1$, $\mcl P^\upa(S^D_+)$ is the set of probability measures with totally ordered support \cite[Proposition~5.9]{issa2024hopflike} and it is not convex. For example when $D = 2$, letting 
\begin{e*}
    A = \begin{pmatrix} 1 & 0 \\ 0 & 0 \end{pmatrix} \quad\text{and}\quad B = \begin{pmatrix} 0 & 0 \\ 0 & 1 \end{pmatrix},
\end{e*}
one can check that $\delta_A$ and $\delta_B$ belong to $\mcl P^\upa(S^2_+)$ but not $(\delta_A + \delta_B)/2$. This prevents any easy generalization of the strict concavity argument from \cite{auffinger2015parisi} to the case of vector spins (i.e. $D >1$).

\newpage

\small
\bibliographystyle{plain}

\begin{thebibliography}{10}

\bibitem{abarra}
Elena Agliari, Adriano Barra, Raffaella Burioni, and Aldo Di~Biasio.
\newblock Notes on the p-spin glass studied via {H}amilton-{J}acobi and smooth-cavity techniques.
\newblock {\em J. Math. Phys.}, 53(6):063304, 29, 2012.

\bibitem{auffinger2015parisi}
Antonio Auffinger and Wei-Kuo Chen.
\newblock The {P}arisi formula has a unique minimizer.
\newblock {\em Comm. Math. Phys.}, 335(3):1429--1444, 2015.

\bibitem{barra2}
Adriano Barra, Gino Del~Ferraro, and Daniele Tantari.
\newblock Mean field spin glasses treated with {PDE} techniques.
\newblock {\em Eur. Phys. J. B}, 86(7):Art. 332, 10, 2013.

\bibitem{barra1}
Adriano Barra, Aldo Di~Biasio, and Francesco Guerra.
\newblock {Replica symmetry breaking in mean-field spin glasses through the Hamilton--Jacobi technique}.
\newblock {\em Journal of Statistical Mechanics: Theory and Experiment}, 2010(09):P09006, 2010.

\bibitem{barra2014quantum}
Adriano Barra, Andrea Di~Lorenzo, Francesco Guerra, and Antonio Moro.
\newblock On quantum and relativistic mechanical analogues in mean-field spin models.
\newblock {\em Proc. R. Soc. A: Math. Phys. Eng. Sci.}, 470(2172):20140589, 2014.

\bibitem{barramulti}
Adriano Barra, Giuseppe Genovese, and Francesco Guerra.
\newblock Equilibrium statistical mechanics of bipartite spin systems.
\newblock {\em J. Phys. A}, 44(24):245002, 22, 2011.

\bibitem{chen2023parisipdevector}
Hong-Bin Chen.
\newblock {P}arisi {PDE} and convexity for vector spins.
\newblock {\em {Preprint, arXiv:2311.10446}}, 2023.

\bibitem{chen2024free}
Hong-Bin Chen.
\newblock On free energy of non-convex multi-species spin glasses.
\newblock {\em {Preprint arXiv:2411.13342}}, 2024.

\bibitem{simulRSB}
Hong-Bin Chen and Jean-Christophe Mourrat.
\newblock Simultaneous replica-symmetry breaking for vector spin glasses.
\newblock {\em {Preprint arXiv:2411.14105}}, 2024.

\bibitem{chenmourrat2023cavity}
Hong-Bin Chen and Jean-Christophe Mourrat.
\newblock On the free energy of vector spin glasses with non-convex interactions.
\newblock {\em Prob. Math. Phys.}, 6(1):1--80, 2025.

\bibitem{chen2022hamilton}
Hong-Bin Chen and Jiaming Xia.
\newblock {Hamilton-Jacobi equations from mean-field spin glasses}.
\newblock {\em {Preprint arXiv:2201.12732}}, 2022.

\bibitem{chen2022fenchelmoreau}
Hong-Bin Chen and Jiaming Xia.
\newblock {F}enchel-{M}oreau identities on convex cones.
\newblock {\em Ann. Fac. sci. Toulouse (6)}, 33(2):287--309, 2024.

\bibitem{HJbook}
Tomas Dominguez and Jean-Christophe Mourrat.
\newblock {\em Statistical mechanics of mean-field disordered systems: a {H}amilton-{J}acobi approach}.
\newblock Zurich Lectures in Advanced Mathematics. EMS Press, 2024.

\bibitem{ET}
I.~Ekeland and R.~Temam.
\newblock {\em Convex analysis and variational problems}.
\newblock North-Holland Publishing Co., Amsterdam-Oxford; American Elsevier Publishing Co., Inc., New York, 1976.

\bibitem{guerra2001sum}
Francesco Guerra.
\newblock Sum rules for the free energy in the mean field spin glass model.
\newblock {\em Fields Institute Communications}, 30(11), 2001.

\bibitem{gue03}
Francesco Guerra.
\newblock Broken replica symmetry bounds in the mean field spin glass model.
\newblock {\em Comm. Math. Phys.}, 233(1):1--12, 2003.

\bibitem{guerra2002}
Francesco Guerra and Fabio~Lucio Toninelli.
\newblock The thermodynamic limit in mean-field spin glass models.
\newblock {\em Comm. Math. Phys.}, 230(1):71--79, 2002.

\bibitem{horn2012matrix}
Roger~A Horn and Charles~R Johnson.
\newblock {\em Matrix analysis}.
\newblock Cambridge university press, 2012.

\bibitem{issa2024permutation}
Victor Issa.
\newblock Existence and uniqueness of permutation-invariant optimizers for {P}arisi formula.
\newblock {\em {Preprint arXiv:2407.13846}}, 2024.

\bibitem{issa2024hopflike}
Victor Issa.
\newblock A {H}opf-like formula for mean-field spin glass models.
\newblock {\em {Preprint arXiv:2410.08754}}, 2024.

\bibitem{mourrat2020nonconvex}
Jean-Christophe Mourrat.
\newblock Nonconvex interactions in mean-field spin glasses.
\newblock {\em Probab. Math. Phys.}, 2(2):281--339, 2021.

\bibitem{mourrat2019parisi}
Jean-Christophe Mourrat.
\newblock The {P}arisi formula is a {H}amilton-{J}acobi equation in {W}asserstein space.
\newblock {\em Canad. J. Math.}, 74(3):607--629, 2022.

\bibitem{mourrat2020free}
Jean-Christophe Mourrat.
\newblock Free energy upper bound for mean-field vector spin glasses.
\newblock {\em Ann. Inst. Henri Poincaré Probab. Stat.}, 59(3):1143--1182, 2023.

\bibitem{mourrat2020extending}
Jean-Christophe Mourrat and Dmitry Panchenko.
\newblock Extending the {P}arisi formula along a {H}amilton-{J}acobi equation.
\newblock {\em Electron. J. Probab.}, 25:Paper No. 23, 17, 2020.

\bibitem{pan}
Dmitry Panchenko.
\newblock {\em The {S}herrington-{K}irkpatrick model}.
\newblock Springer Monographs in Mathematics. Springer, 2013.

\bibitem{pan.multi}
Dmitry Panchenko.
\newblock The free energy in a multi-species {S}herrington-{K}irkpatrick model.
\newblock {\em Ann. Probab.}, 43(6):3494--3513, 2015.

\bibitem{pan.potts}
Dmitry Panchenko.
\newblock Free energy in the {P}otts spin glass.
\newblock {\em Ann. Probab.}, 46(2):829--864, 2018.

\bibitem{pan.vec}
Dmitry Panchenko.
\newblock Free energy in the mixed {$p$}-spin models with vector spins.
\newblock {\em Ann. Probab.}, 46(2):865--896, \noop{2019}2018.

\bibitem{parisi1979infinite}
Giorgio Parisi.
\newblock Infinite number of order parameters for spin-glasses.
\newblock {\em Phys. Rev. Lett.}, 43(23):1754, 1979.

\bibitem{rockafellar1970convex}
R~Tyrrell Rockafellar.
\newblock {\em Convex Analysis}, volume~36.
\newblock Princeton university press, 1970.

\bibitem{Tpaper}
Michel Talagrand.
\newblock The {P}arisi formula.
\newblock {\em Ann. of Math. (2)}, 163(1):221--263, 2006.

\bibitem{talagrand2006parisimeasures}
Michel Talagrand.
\newblock Parisi measures.
\newblock {\em J. Funct. Anal.}, 231(2):269--286, 2006.

\bibitem{villani2}
C\'{e}dric Villani.
\newblock {\em Optimal transport, old and new}, volume 338 of {\em Grundlehren der Mathematischen Wissenschaften}.
\newblock Springer-Verlag, Berlin, 2009.

\end{thebibliography}
\newcommand{\noop}[1]{} \def\cprime{$'$}

\end{document}